\documentclass{amsart}
\usepackage{amsmath}
\usepackage{amsxtra}
\usepackage{amscd}
\usepackage{amsthm}
\usepackage{amsfonts}
\usepackage{amssymb}
\usepackage{mathdots}
\usepackage{bbm}
\usepackage{graphicx}
\usepackage{longtable,booktabs}
\usepackage{cite}
\usepackage{tikz-cd}
\usepackage[bookmarksnumbered, colorlinks, plainpages]{hyperref}
\usepackage[utf8]{inputenc}
\usepackage{color}
\usepackage{enumerate}
\newtheorem{theorem}{Theorem}[section]

\newtheorem{open}[theorem]{Problem}
\newtheorem{lemma}[theorem]{Lemma}

\newtheorem{corollary}[theorem]{Corollary}

\theoremstyle{definition}

\newtheorem{remark}[theorem]{Remark}

\usepackage[paperwidth=17.10cm, paperheight=24.60cm, left=1.5cm, right=1.5cm, top=2cm, bottom=2cm]{geometry}

\usepackage{etoolbox}
\patchcmd{\section}{\scshape}{\large}{}{}
\patchcmd{\subsection}{\bfseries}{\normalfont}{}{}
\patchcmd{\subsubsection}{\itshape}{\normalfont}{}{}

\makeatletter
\def\@setauthors{%
	\begingroup
	\def\thanks{\protect\thanks@warning}%
	\trivlist
	\centering\footnotesize \@topsep30\p@\relax
	\advance\@topsep by -\baselineskip
	\item\relax\normalsize 
	\author@andify\authors
	\def\\{\protect\linebreak}%
	\authors%
	\ifx\@empty\contribs
	\else
	,\penalty-3 \space \@setcontribs
	\@closetoccontribs
	\fi
	\endtrivlist
	\endgroup
}
\def\@settitle{\begin{center}%
		\baselineskip14\p@\relax
		\bfseries\large
		\@title
	\end{center}%
}
\makeatother

\usepackage{titlesec}
\titleformat{\section}
{\normalfont\large\bfseries}
{\thesection}{1em}{}
\titleformat{\subsection}
{\normalfont\large\bfseries}
{\thesubsection}{1em}{}
\titleformat{\subsubsection}
{\normalfont\large\bfseries}
{\thesubsubsection}{1em}{}

\usepackage{tikz,xcolor,hyperref}

\definecolor{lime}{HTML}{A6CE39}
\DeclareRobustCommand{\orcidicon}{%
	\begin{tikzpicture}
		\draw[lime, fill=lime] (0,0) 
		circle [radius=0.16] 
		node[white] {{\fontfamily{qag}\selectfont \tiny ID}};
		\draw[white, fill=white] (-0.0625,0.095) 
		circle [radius=0.007];
	\end{tikzpicture}
	\hspace{-2mm}
}

\newcommand{\orcidSon}{\href{https://orcid.org/0000-0002-9560-6392}{\orcidicon}}
\newcommand{\orcidVinh}{\href{https://orcid.org/0009-0005-9917-1436}{\orcidicon}}

\begin{document}
	\title[On the injectivity of evaluation maps]{On the injectivity of  evaluation maps\\ induced by polynomials on certain algebras}
	\author[Frank Kutzschebauch]{Frank Kutzschebauch\textsuperscript{1,$\S$,*}}
	\author[Tran Nam Son]{Tran Nam Son\textsuperscript{2,3,\ddag}}
	\author[Pham Duy Vinh]{Pham Duy Vinh\textsuperscript{2,3,4,\dag}}
	
	\maketitle

	\begin{center}
		{\small 
			*Corresponding author\\
			\textsuperscript{1}Mathematical Institute,\\ University of Bern,\\ Bern, Switzerland\\
			\textsuperscript{2}Faculty of Mathematics and Computer Science,\\ University of Science,\\ Ho Chi Minh City, Vietnam \\
			\textsuperscript{3}Vietnam National University,\\ Ho Chi Minh City, Vietnam\\
			\textsuperscript{4}Department of Mathematics,\\ Dong Nai University,\\ 9 Le Quy Don Str., Tan Hiep Ward, Bien Hoa City,\\ Dong Nai Province, Vietnam\\
			\textsuperscript{$\S$}frank.kutzschebauch@unibe.ch,\\ \orcidSon \url{https://orcid.org/0000-0003-2272-8687}\\
			\textsuperscript{\ddag}trannamson1999@gmail.com,\\ \orcidSon \url{https://orcid.org/0000-0002-9560-6392}\\
			\textsuperscript{\dag}phamduyvinh.dlu@gmail.com,\\ \orcidVinh\url{https://orcid.org/0009-0005-9917-1436}
		}
	\end{center}
	
	\begin{abstract}
		We explore the injectivity of the evaluation map \( \mathrm{eva}_{f,\mathcal{A}} : \mathcal{A}^m \to \mathcal{A} \), where \( \mathcal{A} \) is an associative algebra over a field \( F \), and \( f \) is a  polynomial in \( m \geq 1 \) variables with coefficients in $F$.  Our investigation reveals that injectivity is possible only when  \( m = 1 \) and \( f \) has degree one; for functions in two or more variables, such injectivity is impossible.
	\end{abstract}
	
	\section{Introduction}
	
	Let $\mathcal{A}$ be an  associative algebra over a field $F$. If $\mathcal{A}$ is unital, then the identity element of $\mathcal{A}$ is denoted by $\mathbf{1}_{\mathcal{A}}$. 	We write $F[x_1, x_2, \ldots, x_m] $ for the unital associative $F$-algebra of all polynomials in \( m \geq1 \) commuting variables $x_1, x_2, \ldots, x_m$ with coefficients in \( F \). In the case $m=1$, we abbreviate $F[x_1, x_2, \dots, x_m]$ as $F[x]$.
	
	For any polynomial \( f \in F[x_1, x_2, \ldots, x_m] \) and any associative \( F \)-algebra \( \mathcal{A} \), we define the evaluation map \( \mathrm{eva}_{f,\mathcal{A}} : \mathcal{A}^m \to \mathcal{A} \) by substituting \( m \) elements from \( \mathcal{A} \) into the variables of \( f \). Note that we implicitly assume that \( f \) has zero constant term when the algebra \( \mathcal{A} \) is non-unital.  Our aim is to determine the conditions under which $\mathrm{eva}_{f, \mathcal{A}}$  is injective. As we know in the case where $m=1$, the function $f$ itself coincides with the sum of its Taylor series in a neighborhood of the expansion point.	A natural starting point for this investigation is the class of polynomials. 
	
	Injectivity is a recurring phenomenon across many branches of mathematics from classical settings to more recent developments. One of the most striking instances where injectivity forces surjectivity is the celebrated Ax–Grothendieck theorem \cite{Pa_Ax_68,Pa_Grothendieck_66}, which the Fields medalist Terence Tao has described as one of his favorite algebraic results. This theorem connects naturally to the spirit of the Jacobian Conjecture \cite{Bo_Es_00}, a famous open problem that continues to draw intense interest.	 In the context of Banach space theory, a striking recent result by Antonio Avilés and Piotr Koszmider demonstrates that there exists an infinite-dimensional Banach space of the form \( C(K) \), where $K$ is a compact Hausdorff space, such that every injective operator \( T: C(K) \to C(K) \) is automatically surjective; see~\cite{Pa_Av_13}.
	
	In what follows, the paper begins in Section~\ref{section base} with an analysis of the case $\mathcal{A}=F$. Some of the results obtained there serve as a foundation for Section~\ref{nilpotent}, where we address the situation in which $\mathcal{A}$ contains idempotent or nilpotent elements. Since matrix algebras provide a concrete example of algebras with nilpotent elements, it follows that Section~\ref{matrix section} is then devoted to examining this phenomenon in that setting.
	
	Besides, our approach not only addresses the core problem under consideration but also provides deeper insight into some recent developments in the literature. For example, Theorem~\ref{inj poly} in this paper offers a substantial generalization of \cite[Lemma~4.2]{Pa_BuReSa_19} and \cite[Proposition~2.1]{Pa_BuReSa_21}, though certain ideas in our proof are inspired by theirs. In their work, these results are used to show that if the evaluation map \( \mathrm{eva}_{f, \mathrm{M}_n(\mathbb{C})} \) is injective for \( n \in \{2,3\} \), then it must also be surjective. Our findings go further: we show in Theorem~\ref{variable1} that any such injective map must in fact be a polynomial of degree $1$, which immediately implies surjectivity in this context.
	
	\section{On base fields}\label{section base}
	
	We begin with the following remark.
	
	\begin{remark} \label{remark}
		Let $\mathcal{A}$ be an associative algebra over a field $F$ and let $f\in F[x_1,x_2,\ldots,x_m]$.
		\begin{enumerate}[\rm (i)]
			\item If $\mathrm{eva}_{f,\mathcal{A}}$ is injective, then $f$ must be nonconstant. 
			\item If $f\in F[x]$ has the form: $f=ax+b$ where $a,b\in F$ and $a\neq0$, then $\mathrm{eva}_{f,\mathcal{A}}$ is injective.
			\item If $\mathcal{A}$ is unital and $\mathrm{eva}_{f,\mathcal{A}}$ is injective, then $\mathrm{eva}_{f,F}$ is also injective.
		\end{enumerate}
	\end{remark}
	
	From item (iii) of Remark~\ref{remark}, we continue with the case where the algebra is simply the field itself, i.e., $\mathcal{A} = F$. A quick counterexample is  the polynomial $f(x) = x^2$. Over any field $F$ of characteristic different from $2$, we observe that $f(1) = 1 = f(-1)$, even though $1 \ne -1$. Hence, the evaluation map $\mathrm{eva}_{f, F}$ fails to be injective in this case.	When the field $F$ is algebraically closed,  the story sharpens and a stronger statement holds, as follows:
	
	\begin{theorem}\label{al1}
		Let \( F \) be an algebraically closed field, and let $f\in F[x_1,x_2,\ldots,x_m]$. Then, the evaluation map \( \mathrm{eva}_{f, F} \) is injective if and only if $f\in F[x]$ has degree $1$.
	\end{theorem}
	
	\begin{proof}
		By Remark~\ref{remark}, it suffices to show that if the evaluation map \( \mathrm{eva}_{f, F} \) is injective, then $f$ must have degree $1$ in one variable.  
		
		We begin with the base case where $f$ is a polynomial in one variable. Observe that any polynomial $f$ with degree greater than $1$ can be factored into linear factors (possibly repeated). If it is a single repeated factor, say $f(x) = a(x - b)^n$, where $a,b\in F$, $a \ne 0$ and $n\geq2$, then distinct inputs symmetrically placed around $b$ will map to the same nonzero value, that is, it is easy to find two distinct values of $x$ that map to the same nonzero value under $f$. If it splits into several distinct roots, say $$f(x)=a(x-b_1)^{m_1}(x-b_2)^{m_2}\cdots(x-b_k)^{m_k},$$ where $a \ne 0$, $m_i \ge 1$, $b_i \ne b_j$ for $i \ne j$, and $\displaystyle\sum_{i=1}^k m_i = n$, then  already $f(b_i) = f(b_j) = 0$ for $i \ne j$,  even though $b_i \ne b_j$. In either scenario, we always find two different elements sent to the same result, so injectivity fails unless $f$ has degree $1$.
		
		Lastly, we consider the case in which $f$ is a polynomial in at least two variables.  By Remark~\ref{remark}, if the evaluation map \( \mathrm{eva}_{f,F} \) is injective, then \( f \) must be a nonconstant polynomial. This implies that there exists an index \( i \in \{1, 2, \ldots, m\} \) such that the degree of \( x_i \) in \( f \) is at least 1. We can write \( f \) as a polynomial in \( x_i \), keeping the other variables \( x_1, x_2, \ldots, x_{i-1}, x_{i+1}, \ldots, x_m \) fixed: 
		$$f = \sum_{j=0}^k g_j x_i^j,$$
		where \( g_j \in F[x_1, x_2, \ldots, x_{i-1}, x_{i+1}, \ldots, x_m] \), and \( k \) is the degree of \( f \) with respect to \( x_i \). Since \( \mathrm{eva}_{f,F} \) is injective, the polynomial \( f \) is still nonconstant in \( x_i \). Because \( F \) is algebraically closed, we know that \( f \), as a polynomial in \( x_i \), must have a root in \( F \). Since the other variables \( x_1, x_2, \ldots, x_{i-1}, x_{i+1}, \ldots, x_m \) are fixed, it follows that \( f \) has infinitely many distinct roots in \( F^m \). Therefore, the evaluation map \( \mathrm{eva}_{f,F} \) cannot be injective, which leads to a contradiction, as promised.
	\end{proof}
	
	\begin{theorem}\label{variable}
		Let \( F \) be a field, and let \( m \geq 2 \) be an integer. If \( F \) is either a finite field or the field $\mathbb{R}$ of all real numbers, then for any polynomial \( f \in F[x_1, x_2, \ldots, x_m] \), the evaluation map \( \mathrm{eva}_{f,F} \) cannot be injective.
	\end{theorem}
	
	\begin{proof}
		First, let us consider the case where \( F \) is a finite field. In this case, \( F^m \) is a finite set. If the evaluation map \( \mathrm{eva}_{f,F} \) were injective, the number of elements in \( F^m \) would have to be less than or equal to the number of elements in \( F \), which implies \( m = 1 \). This leads to a contradiction with  the assumption that $m\geq2$, so there is no polynomial whose evaluation map is injective in this case.
		
		We now turn to the case where \( F = \mathbb{R} \).	 Assume, for contradiction, that the evaluation map $\mathrm{eva}_{f,\mathbb{R}} : \mathbb{R}^m \to \mathbb{R}$ is injective. Restrict it to the unit sphere $S^{m-1} \subseteq \mathbb{R}^m$ and denote
		\[
		g := \mathrm{eva}_{f,\mathbb{R}}|_{S^{m-1}} : S^{m-1} \to \mathbb{R}.
		\]
		Since $f$ is a polynomial, it follows that $g$ is continuous, and since $\mathrm{eva}_{f,\mathbb{R}}$ is injective, so is $g$. The domain $S^{m-1}$ is compact and connected, hence the image 
		$K := g(S^{m-1})$ is also compact and connected in $\mathbb{R}$. Thus $K$ must be either a single point or a closed interval $[a,b]$.
		
		If $K$ is a point, then $g$ is constant, contradicting injectivity (see Remark~\ref{remark}). 
		If $K=[a,b]$, then $g : S^{m-1} \to [a,b]$ is a continuous bijection from a compact space to a Hausdorff space, hence a homeomorphism. 
		Pick $y \in (a,b)$ and let $x = g^{-1}(y)$. Then
		\[
		S^{m-1}\setminus \{x\} \;\cong\; [a,b]\setminus \{y\}.
		\]
		But for $m\geq 2$, the set $S^{m-1}\setminus\{x\}$ is connected 
		(indeed, homeomorphic to $\mathbb{R}^{m-1}$), 
		whereas $[a,b]\setminus\{y\}$ is disconnected. 
		This contradiction shows that no such polynomial can exist.
	\end{proof}

	\begin{remark}\label{remark1}
		Let $F$ be a field and let $f \in F[x]$. 
		\begin{enumerate}[\rm (i)]
			\item Suppose $F = \mathbb{F}_q$ is a finite field with $q$ elements. It is  known (see, for instance, \cite[7.1 Lemma, p.~348]{Bo_Li_97}) that the following conditions are equivalent:
			\begin{enumerate}[\rm (a)]
				\item The evaluation map $\mathrm{eva}_{f,F}$ is injective. 
				\item The evaluation map $\mathrm{eva}_{f,F}$ is surjective.
				\item The polynomial $f$ is a permutation polynomial, i.e. the induced map $\mathrm{eva}_{f,F}$ defines a permutation of $F$. 
				A characterization of permutation polynomials is given by Hermite’s criterion (see \cite[7.4 Theorem]{Bo_Li_97}).
			\end{enumerate}
			
			\item If $F = \mathbb{R}$, then $\mathrm{eva}_{f,F}$ is injective precisely when $f$ is strictly monotone on the whole real line. 
			In particular, this can only happen when the degree of $f$ is odd.
		\end{enumerate}
	\end{remark}

	\section{Algebras containing idempotent or nilpotent elements}\label{nilpotent}

	Theorem~\ref{al1} leads to the following theorem, which addresses the injectivity of the evaluation map on algebras over algebraically closed fields.
	
	\begin{theorem}\label{variable1}
		Let \( \mathcal{A} \) be an associative algebra over a field \( F \), and let \( m \geq 1 \) be an integer. Suppose that \( f \in F[x_1, x_2, \ldots, x_m] \), and if \( \mathcal{A} \) is non-unital, then the polynomial \( f \) has zero constant term and \( \mathcal{A} \) contains a nonzero idempotent. 
		\begin{enumerate}[\rm (i)]
			\item For \(m=1\) and depending on the ground field \(F\), the necessary conditions for \(\mathrm{eva}_{f,\mathcal{A}}\) to be injective can be summarized in the following table,  with the corresponding conditions on the polynomial \(f\).
			
			\begin{longtable}{p{0.45\textwidth} p{0.45\textwidth}}
				\toprule
				\textbf{Field \(F\)} & \textbf{Necessary condition on \(f\)} \\
				\midrule
				\(F\) is algebraically closed. & \(f \in F[x]\) has degree \(1\). \\
				\(F\) is a finite field.          & \(f\in F[x]\) is a permutation polynomial. \\
				\(F\) is the field of real numbers.      & $f\in F[x]$ has odd degree. \\
				\bottomrule
			\end{longtable}
			\item 	If $m\geq2$, and \( F \) is a finite field, the field of real numbers, or an algebraically closed field, then the evaluation map \( \mathrm{eva}_{f,\mathcal{A}} \) cannot be injective.
		\end{enumerate}
	\end{theorem}
	
	\begin{proof}
		We begin with the following observations:
		\begin{enumerate}
			\item If  \( \mathcal{A} \) is unital and the evaluation map  \( \mathrm{eva}_{f,\mathcal{A}} \) is injective, then the evaluation map  \( \mathrm{eva}_{f,F} \) is also injective. Indeed, let \( \alpha = (\alpha_1, \ldots, \alpha_m) \) and \( \beta = (\beta_1, \ldots, \beta_m) \) be elements in \( F^m \) such that \( f(\alpha) = f(\beta) \). Evaluate \( f \) at the corresponding scalars  to obtain:
			\[
			f(\alpha_1 \mathbf{1}_\mathcal{A}, \ldots, \alpha_m \mathbf{1}_\mathcal{A}) = f(\alpha) \mathbf{1}_\mathcal{A} = f(\beta) \mathbf{1}_\mathcal{A} = f(\beta_1 \mathbf{1}_\mathcal{A}, \ldots, \beta_m \mathbf{1}_\mathcal{A}).
			\]
			Since we assumed the evaluation map  \( \mathrm{eva}_{f,\mathcal{A}} \) is injective, we must have:
			\[
			(\alpha_1 \mathbf{1}_\mathcal{A}, \ldots, \alpha_m \mathbf{1}_\mathcal{A}) = (\beta_1 \mathbf{1}_\mathcal{A}, \ldots, \beta_m \mathbf{1}_\mathcal{A}),
			\]
			which implies \( \alpha_i = \beta_i \) for all \( i \). Thus, \( \alpha = \beta \), showing that \( \mathrm{eva}_{f,F} \) is injective.
			\item If \( \mathcal{A} \) is non-unital and \( \mathcal{A} \) contains a nonzero idempotent element \( e \), then the corner subalgebra \( e \mathcal{A} e = \{ eae \mid a \in \mathcal{A} \} \) forms a unital subalgebra with identity element \( e \).	If \( \mathrm{eva}_{f,\mathcal{A}} \) is injective, then the restriction \( \mathrm{eva}_{f,e\mathcal{A}e} \) is also injective. Indeed, if \( \alpha, \beta \in (e\mathcal{A}e)^m \) satisfy \( f(\alpha) = f(\beta) \), then injectivity of \( \mathrm{eva}_{f,\mathcal{A}} \) implies \( \alpha = \beta \). This means that the restriction \( \mathrm{eva}_{f,e\mathcal{A}e} \) is also injective. But \( e\mathcal{A}e \) is a unital algebra, and Observation 1 has already shown how injectivity behaves in that setting.
		\end{enumerate} By combining these observations with Theorems~\ref{al1}, \ref{variable} and Remark~\ref{remark1}, the proof is completed.
	\end{proof}

	To that end, we now adopt a structured approach, drawing inspiration from the ideas found in  \cite[Lemma 4.2]{Pa_BuReSa_19}. Our version, however, includes modifications adapted to the specific context of our study, which we outline in detail below.
	
	\begin{theorem}\label{inj poly}
		Let $\mathcal{A}$ be a unital associative algebra over a field $F$ containing a nonzero nilpotent element in $\mathcal{A}$ with nilpotency index \(2\) and let $f \in F[x]$ be a polynomial. If the evaluation map $\operatorname{eva}_{f, \mathcal{A}}$ is injective, then for every $\lambda \in F$, the polynomial $f - \lambda$ has only simple roots in $F$.
	\end{theorem}
	
	\begin{proof}
		Fix an arbitrary \(\lambda \in F\). If \(b \in F\) is a root of \(f(x) - \lambda\) of multiplicity \(k\), we may factor $f(x) - \lambda = (x - b)^k q(x)$,	where $q(x)\in F[x]$ such that \(q(b) \neq 0\). Our goal is to show that \(k = 1\). Since \(\mathcal{A}\) contains a nonzero nilpotent element \(e\) with \(e^2 = 0\), we can define
		$v = b  \mathbf{1}_{\mathcal{A}}$, and  $u = b \mathbf{1}_{\mathcal{A}} + e$. At \(u\), we have \((u - b  \mathbf{1}_{\mathcal{A}})^k = e^k = 0\) if \(k \ge 2\). Thus,
		$f(u) - \lambda  \mathbf{1}_{\mathcal{A}} = (u - b  \mathbf{1}_{\mathcal{A}})^k q(u) = 0.$ At \(v\), since \(v = b  \mathbf{1}_{\mathcal{A}}\), it follows that
		$f(v) - \lambda  \mathbf{1}_{\mathcal{A}} = 0.$ Therefore, both \(u\) and \(v\) are mapped to \(\lambda \mathbf{1}_{\mathcal{A}}\). The assumption that the evaluation map is injective forces \(u = v\), which implies \(e = 0\), contradicting our assumption that \(e\) is a nonzero nilpotent. This contradiction shows that \(k\) must be \(1\), as claimed.
	\end{proof}
	
	As far as we can tell, the collection of unital associative algebras admitting a nonzero nilpotent of index two seems to be rather broad. Matrix algebras offer a concrete instance of this phenomenon. This leads us naturally to the focus of the next section.

	\section{Matrix algebras over fields}\label{matrix section}
	
	As shown in Proposition~\ref{al1}, when \( F \) is an algebraically closed field and \( f \in F[x] \) is nonconstant, a necessary and sufficient condition for the evaluation map \( \mathrm{eva}_{f,F} \) to be injective is that \( f \) has degree one. Since \( F \) can be naturally identified with \( \mathrm{M}_1(F) \), this observation suggests that a similar criterion may hold for matrix algebras of higher dimension. With this in mind, we establish the following theorem as a step toward extending the result to arbitrary fields. We use the standard notation \( \mathrm{M}_n(F) \) to denote the algebra of \( n \times n \) matrices with entries in a field \( F \).

	\begin{theorem}\label{matrix}
		Let $F$ be any field and let $f\in F[x]$ have degree greater than $1$.  Set $c = f(0),$ and	$g(x) = f(x) - c.$
		Then, $\deg g = \deg f > 1$ and $g(0)=0$, so we factor $g(x) = x^m\,h(x),$
		where $m\ge1$ is maximal and $h(x)\in F[x]$ such that $h(0)\neq0$.  For any integer $n\geq2$, define the evaluation map
		$\mathrm{eva}_{f,\mathrm{M}_n(F)}:\;\mathrm{M}_n(F)\;\longrightarrow\;\mathrm{M}_n(F),$ given by
		$\mathrm{eva}_{f,\mathrm{M}_n(F)}(A)=f(A).$
		\begin{enumerate}[{\upshape(i)}]
			\item If $m\ge2$, then $\mathrm{eva}_{f,\mathrm{M}_n(F)}$ fails to be injective.
			\item If $m=1$, write $g(x)=x\,h(x)$, where $h(x)\in F[x]$, and factor
			\[
			h(x)=q_1(x)^{e_1}q_2(x)^{e_2}\cdots q_r(x)^{e_r}
			\]
			into irreducible factors $q_i(x)$'s in $F[x]$, where $r,e_1,e_2,\ldots,e_r$ are positive integers.  Let
			$\displaystyle d = \min_{1\le i\le r}\bigl(\deg q_i\bigr),$
			and choose one factor $q(x)$ of degree $d$.  Then:
			\begin{enumerate}[{\upshape(a)}]
				\item If $n\ge d$, then $\mathrm{eva}_{f,\mathrm{M}_n(F)}$ is again not injective.
				\item If $n<d$, then for every nonzero $A\in \mathrm{M}_n(F)$, one has
				$\mathrm{eva}_{f,\mathrm{M}_n(F)}(A)\neq\mathrm{eva}_{f,\mathrm{M}_n(F)}(0).$
			\end{enumerate}
		\end{enumerate}
	\end{theorem}
	
	\begin{proof}
		(i)\quad If $m\ge2$, pick any nonzero nilpotent matrix $N\in \mathrm{M}_n(F)$ with $N^2=0$ (for instance, a single $2\times2$ Jordan block in the upper-left).  Then, $g(N)=N^m\,h(N)=0,$
		so $f(N)=g(N)+c\mathrm{I}_n=c\mathrm{I}_n=f(0)\mathrm{I}_n$, where $\mathrm{I}_n=\mathbf{1}_{\mathrm{M}_n(F)}$.
		Thus $\mathrm{eva}_{f,\mathrm{M}_n(F)}(N)=\mathrm{eva}_{f,\mathrm{M}_n(F)}(0)$ despite $N\neq0$, proving non-injectivity.
		
		(ii)\quad Now $g(x)=x\,h(x)$ with $h(0)\neq0$.  Factor $h(x)$ as above and recall that $d$ is the smallest degree of its irreducible factors.
		
		\begin{enumerate}[{\upshape(a)}]
			\item If $n\ge d$, let $C\in \mathrm{M}_d(F)$ be the companion matrix of $q(x)$, so $q(C)=0$ but $C\neq0$.  We embed it into $\mathrm{M}_n(F)$ by
			$C'=\begin{pmatrix}C&0\\0&0_{\,n-d}\end{pmatrix}$, and write $h(x)=q(x)\,k(x)$, where $k(x)\in F[x]$, gives
			$g(C')=C'\,h(C')=\begin{pmatrix}C\,q(C)\,k(C)&0\\0&0\end{pmatrix}=0,$			hence $f(C')=c\mathrm{I}_n$, so $\mathrm{eva}_{f,\mathrm{M}_n(F)}(C')=\mathrm{eva}_{f,\mathrm{M}_n(F)}(0)$ with $C'\neq0$, proving also non-injectivity.
			
			\item If $n<d$, suppose for contradiction there is $A\neq0$ with $\mathrm{eva}_{f,\mathrm{M}_n(F)}(A)=\mathrm{eva}_{f,\mathrm{M}_n(F)}(0)$, i.e.\ $f(A)=c\mathrm{I}_n$, so $g(A)=0$.  Let $m_A(x)\in F[x]$ be the minimal polynomial of $A$.  Then $\deg m_A\le n<d$, so no irreducible factor of $g$ divides $m_A$, and thus $\gcd(m_A,g)=1$.  By Bézout’s identity, there exist $U,V\in F[x]$ with
			$U(x)\,m_A(x)+V(x)\,g(x)=1.$
			Substituting $x=A$ yields
			$U(A)\,m_A(A)+V(A)\,g(A)=\mathrm{I}_n,$
			but $m_A(A)=0$ and $g(A)=0$, which is a contradiction.  Hence no nonzero $A$ can satisfy $f(A)=c\mathrm{I}_n$.
		\end{enumerate}
		This completes the proof in both scenarios, giving a precise criterion for when the evaluation map can or cannot be injective.
	\end{proof}
	
	We pose the following open problem, which arises from Theorem~\ref{matrix}.
	
	\begin{open}\label{open1}
		Does injectivity hold for $\mathrm{eva}_{f,\mathrm{M}_n(F)}$ in the case $m=1$ and $n<d$ in Theorem~\ref{matrix}?
	\end{open}

	\begin{remark}\label{remark2}
		To illustrate Problem~\ref{open1}, consider the polynomial $f(x)=x^4+2x+c \in \mathbb{Q}[x]$, where $\mathbb{Q}$ is the field of rational numbers. Let $$A=\begin{pmatrix}0&\tfrac{1}{2}\\[2pt] 1&-1\end{pmatrix}\oplus 0_{n-2},\text{ and }B=\begin{pmatrix}0&-\tfrac{3}{2}\\[2pt] 1&1\end{pmatrix}\oplus 0_{n-2},$$
		where \(\oplus\) denotes the block direct sum. A straightforward computation 
		shows that $f(A)=\frac{3+4c}{4}\mathrm{I}_2 \oplus c\mathrm{I}_{n-2}
		= f(B),$
		while \(A\neq B\). Hence, the evaluation map $\mathrm{eva}_{f,\mathrm{M}_n(\mathbb{Q})}$ fails to be injective.  	This example is worth emphasizing. Indeed, writing $f(x)-f(0)=x^4+2x = x(x^3+2),$
		we see that the decomposition involves \(m=1\) and \(d=3\), with 
		\(h(x)=x^3+2\) irreducible in \(\mathbb{Q}[x]\). Thus, the non-injectivity 
		already appears in the case \(n=2<d\), exactly the range not covered by 
		Theorem~\ref{matrix}.
		
		By the way, the evaluation map $\mathrm{eva}_{f,\mathbb{Q}}$ also fails to be injective.  Indeed, suppose that there exist distinct rationals $x,y\in\mathbb{Q}$ with $f(x)=f(y)$. Then 
		\[
		f(x)-f(y)=(x-y)\big((x+y)(x^2+y^2)+2\big)=0,
		\]
		so necessarily
		\begin{equation}
			(x+y)(x^2+y^2)=-2. \label{a*}
		\end{equation}	Write $x=\frac{a}{u}$ and $y=\frac{b}{u}$ with integers $a,b,u$ such that the greatest common divisor of $a,b,u$ is $1$. 
		Clearing denominators in \eqref{a*} gives
		\begin{equation}
			(a+b)(a^2+b^2)=-2u^3. \label{a**}
		\end{equation}	A parity check shows that $a$ and $b$ cannot be of opposite parity, 
		for otherwise $(a+b)(a^2+b^2)$ would be odd, contradicting the right-hand side. 
		Thus both $a,b$ are either even or odd. Since the greatest common divisor of $a,b,u$ is $1$, it follows that $a$ and $b$ must both be odd.	 In this case, $a+b \equiv 2 \pmod{4},$ and $a^2+b^2 \equiv 2 \pmod{4},$
		so the $2$-adic valuation $v_2(a+b)=v_2(a^2+b^2)=1$. Hence $v_2\!\big((a+b)(a^2+b^2)\big)=2.$	On the other hand, from \eqref{a**}, it is shown that
		$v_2(-2u^3)=1+3v_2(u),$
		which can never equal $2$. This contradiction shows that no such pair $(x,y)$ exists, as claimed.
	\end{remark}
	
	From Remark~\ref{remark2}, we can also formulate another open problem:
	
	\begin{open}
		Is there an irreducible polynomial $f$ over $\mathbb{Q}$ of degree greater than $2$ such that the evaluation map $\mathrm{eva}_{f,\mathrm{M}_2(\mathbb{Q})}$ is injective?
	\end{open}

	However, following Theorem~\ref{matrix} again,  when \( d \in \{1,2\} \), this issue does not arise, since the condition \( n \geq d \) is always satisfied for \( n \geq 2 \). A fundamental example is any algebraically closed field, where every nonconstant polynomial splits into linear factors, hence all irreducible polynomials are of degree one. Using Theorem~\ref{matrix}, we arrive at the following conclusion regarding matrix algebras over algebraically closed fields: if \( F \) is an algebraically closed field and \( f \in F[x] \) is a polynomial of degree greater than one, then for every integer \( n \geq 1 \), the evaluation map \( \mathrm{eva}_{f,\mathrm{M}_n(F)} \) is not injective. Interestingly, this statement coincides exactly with a special case of Theorem~\ref{variable1}, which can be derived directly from Proposition~\ref{al1} without invoking Theorem~\ref{matrix}.

	Another classical example where all irreducible polynomials have degree at most two is the field of real numbers \( \mathbb{R} \). More broadly, it is known that real closed fields are precisely those in which every irreducible polynomial over the field has degree either one or two. Therefore, for any real closed field, Theorem~\ref{matrix} applies fully, and we may deduce the following corollary.
	
	\begin{corollary}\label{real closed}
		Let \( F \) be a real closed field, and let \( f \in F[x] \) be a polynomial of degree greater than $1$. Then, for every integer \( n \geq 2 \), the evaluation map \( \mathrm{eva}_{f,\mathrm{M}_n(F)} \) fails to be injective.
	\end{corollary}
	
	It is easy to see that the conclusion of Corollary~\ref{real closed} does not necessarily hold when \( n = 1 \). For instance, the polynomial \( f = x^3 \) serves as a concrete counterexample: the corresponding evaluation map \( \mathrm{eva}_{f,F} \) is injective, even though \( f \) has degree 3.

	Drawing upon Theorems~\ref{variable1} and~\ref{real closed}, we are led naturally to the following corollary.
	
	\begin{corollary}\label{ss}
		Let \( \mathcal{A} \) be an associative algebra over a field \( F \), and let \( f \in F[x] \) be a polynomial of degree greater than one. Fix a positive integer \( n \). Suppose \( K \) is either a real closed field (when \( n \geq 2 \)) or an algebraically closed field (when \( n \geq 1 \)) such that \( F \subseteq K \subseteq \mathcal{A} \). Additionally, if \( A \) is non-unital, assume that \( f \) has zero constant term. Then, the evaluation map \( \mathrm{eva}_{f,\mathrm{M}_n(\mathcal{A})} \) is not injective.
	\end{corollary}
	
	\begin{proof}
		Injectivity of the map \( \mathrm{eva}_{f,\mathrm{M}_n(\mathcal{A})} \) would imply that \( \mathrm{eva}_{f,\mathrm{M}_n(K)} \) is also injective, which contradicts Theorem~\ref{variable1} and Corollary~\ref{real closed}. Therefore, the proof is complete.
	\end{proof}

	Returning to Theorem~\ref{matrix}, we note that when \( n < d \), difficulties arise due to the limitation on matrix size. However, this issue can be circumvented by working within the framework of the finitary matrix algebra. For clarity, we briefly recall its construction. Let \( F \) be a field. For each positive integer \( n \), there is a natural embedding of \( \mathrm{M}_n(F) \) into \( \mathrm{M}_{n+1}(F) \) that sends a matrix \( A = (a_{i,j}) \) to the matrix \( A' \), obtained by appending a row and a column of zeros, i.e., \( A' = A \oplus 0 \) where $0\in\mathrm{M}_1(F)$. This yields an ascending chain of matrix algebras:
	\[
	\mathrm{M}_1(F) \subseteq \mathrm{M}_2(F) \subseteq \cdots \subseteq \mathrm{M}_n(F) \subseteq \cdots.
	\]
	The union of this chain, denoted \( \mathrm{M}_\infty(F) \), is called the \emph{finitary matrix algebra} over \( F \):
	\[
	\mathrm{M}_\infty(F) = \bigcup_{n \geq 1} \mathrm{M}_n(F).
	\]
	Each element of \( \mathrm{M}_\infty(F) \) is a countably infinite matrix with only finitely many nonzero entries. In particular, every such matrix can be written in the form \( A \oplus 0 \) for some \( A \in \mathrm{M}_n(F) \), with \( n \geq 1 \), and where \( 0 \) denotes the infinite zero matrix. Without loss of generality, we may always assume \( n \geq 2 \). With this setup in place, we now establish the following result concerning finitary matrix algebras. 	Note that $\mathrm{M}_\infty(F)$ is non-unital, so we restrict our attention to polynomials whose constant term is zero. Note that this result follows immediately from Theorem~\ref{variable1} when the field \( F \) is algebraically closed, and more generally, it can be derived from Theorem~\ref{matrix}.

	\begin{corollary}\label{finitary1}
		Let \( F \) be a field, and let \( f \in F[x] \) be a polynomial of degree greater than one and with zero constant term. Then, the evaluation map \( \operatorname{eva}_{f, \mathrm{M}_\infty(F)} \) is not injective.
	\end{corollary}
	
	\begin{proof}
		Take any matrix \( A \in \mathrm{M}_\infty(F) \). Since \( A \) has only finitely many nonzero entries, we may write \( A = A' \oplus 0 \) for some \( A' \in \mathrm{M}_n(F) \), where \( n \geq 1 \). If the map \( \operatorname{eva}_{f, \mathrm{M}_\infty(F)} \) were injective, then its restriction to \( \mathrm{M}_n(F) \) would also be injective. However, this contradicts Theorem~\ref{matrix}, as \( n \) can be chosen large enough to ensure that it exceeds the value of \( d \) mentioned in Theorem~\ref{matrix}. Therefore, \( \operatorname{eva}_{f, \mathrm{M}_\infty(F)} \) cannot be injective.
	\end{proof}
	
	In what follows, we continue our exploration of injective maps between algebras, particularly those that are isomorphic to matrix algebras. 	Before proceeding, we require the following technical lemma. While it may seem straightforward, we include it here for clarity and to highlight its specific role in the proof.
	
	\begin{lemma}\label{lem}
		Let \( \mathcal{A} \) and \( \mathcal{B} \) be unital associative algebras over a field \( F \), and suppose they are isomorphic as \( F \)-algebras via an algebra homomorphism \( \varphi: \mathcal{A} \to \mathcal{B} \). Let \( f \in F[X] \) be a polynomial.   Then,  \( \mathrm{eva}_{f,\mathcal{A}} \) is injective if and only if \( \mathrm{eva}_{f,\mathcal{B}} \) is injective.
	\end{lemma}
	
	In light of Lemma~\ref{lem}, our results at this stage may apply to any algebra \( \mathcal{A} \) that is isomorphic to \( \mathrm{M}_n(\mathcal{A}) \) for some integer \( n \geq 2 \) (or weaker forms such as direct products of matrix algebras). This condition defines a remarkably broad class of algebras. For example, the Wedderburn--Artin theorem tells us that any finite-dimensional semisimple algebra over a field is isomorphic to a finite direct product of matrix algebras over division rings. Additional examples arise from functional analysis: consider the algebra \( B(X) \) of all bounded linear operators on a Banach space \( X \). If \( X \cong X \oplus X \), then \( B(X) \cong \mathrm{M}_2(B(X)) \). In particular, infinite-dimensional Hilbert spaces \( H \) satisfy \( H \cong H \oplus H \), so this property holds for most Banach spaces. In fact, the first known infinite-dimensional Banach space not isomorphic to its square was only constructed in 1960~\cite{Pa_BP_60}. A similar phenomenon occurs in algebra: if \( X \) is a free module of infinite rank over a ring \( R \), then its endomorphism ring \( \mathrm{End}_R(X) \) is isomorphic to \( \mathrm{M}_2(\mathrm{End}_R(X)) \), since \( X \cong X \oplus X \) in this setting as well.
	
	We conclude this section with a result concerning functions of several variables that are naturally associated with polynomials in one variable over a field \( F \). The motivation arises from observing that \( \mathrm{M}_n(F) \), the algebra of \( n \times n \) matrices over \( F \), is isomorphic to the \( F \)-vector space \( F^{n^2} \) via the linear map
	\[
	A = (a_{i,j}) \mapsto \varphi(A) = (a_{1,1}, a_{1,2}, \ldots, a_{1,n}, a_{2,1}, \ldots, a_{2,n}, \ldots, a_{n,1}, \ldots, a_{n,n}).
	\]
	However, it is important to emphasize that this map is not a ring homomorphism, since a ring homomorphism from \( \mathrm{M}_n(F) \) to $F$ would send all nilpotent matrices to zero, which is clearly not the case here. Thus, the assumption that \( \varphi \) is an algebra homomorphism in Lemma~\ref{lem} plays a crucial role.

	Now let \( f \in F[x] \) be a polynomial, and let \( n\geq2 \) be an integer.  We proceed by recalling the earlier remark that \( \mathrm{M}_n(F) \cong F^{n^2} \) as vector spaces. Define a map
	\[
	\overline{f} : F^{n^2} \xrightarrow{\theta} \mathrm{M}_n(F) \xrightarrow{\mathrm{eva}_{f, \mathrm{M}_n(F)}} \mathrm{M}_n(F) \xrightarrow{\varphi} F^{n^2},
	\]
	where the map \( \theta \) is defined by arranging the entries of \( F^{n^2} \) into an \( n \times n \) matrix: given \( (a_1, \ldots, a_{n^2}) \in F^{n^2} \), we set \( \theta(a_1, \ldots, a_{n^2}) = (b_{i,j}) \), where \( b_{i,j} = a_{(i-1)n + j} \). It is straightforward to check that \( \theta \) is bijective and in fact the inverse of \( \varphi \), i.e., \( \theta = \varphi^{-1} \).
	
	It follows that if \( \overline{f} \) is injective, then so is the evaluation map \( \mathrm{eva}_{f, \mathrm{M}_n(F)} \), due to the bijectivity of both \( \theta \) and \( \varphi \). Now, invoking Theorem~\ref{matrix}, we know that if \( n \) is sufficiently large, then the injectivity of \( \mathrm{eva}_{f, \mathrm{M}_n(F)} \) implies that \( f \) must be linear. Similarly, even when \( n \) is not large, the same conclusion holds under the assumption that \( F \) is algebraically closed or real closed, as established in Proposition~\ref{al1} and Corollary~\ref{real closed}.
	
	This discussion leads us to the following theorem.
	
	\begin{theorem}\label{Ax}
		If the associated map \( \overline{f} \), as defined above, is injective, then \( f \) must have degree one in each of the following cases:
		\begin{enumerate}[\rm (i)]
			\item \( n \) is sufficiently large,
			\item \( F \) is algebraically closed,
			\item \( F \) is real closed.
		\end{enumerate}
	\end{theorem}

	As is well known, every complex polynomial is in fact an entire function on the complex plane. 
	This naturally leads us to ask what can be said about the injectivity of entire functions. 
	Recall that in complex analysis, an \textit{entire function} in one variable is a complex-valued function 
	that is holomorphic on all of $\mathbb{C}$. 
	Moreover, any entire function $f:\mathbb{C}\to\mathbb{C}$ admits a power series expansion $\displaystyle f(z) = \sum_{i=0}^{\infty} a_i z^i,$
	which converges everywhere on $\mathbb{C}$. 
	
	It is a classical fact that every injective entire function must necessarily be  of the form $f(z) = az + b$ with $a \neq 0$ and $a, b \in \mathbb{C}$; 
	see \cite[Chapter 10, Section 2, Theorem, p.~311]{Bo_Re_91} for details. 
	Since the radius of convergence of $f$ is infinite, its defining power series  $\displaystyle  f(z) = \sum_{i=0}^\infty a_i z^i$
	remains valid when substituting a matrix $A$ over $\mathbb{C}$, leading to the natural definition  
	$\displaystyle  f(A) = \sum_{i=0}^\infty a_i A^i,$
	as discussed in \cite[Theorem 1]{Bo_FrSi_08}. 
	Following a similar line of reasoning, one can extend this framework to the setting of unital associative Banach algebras. 
	More precisely, if $\mathcal{A}$ is such an algebra and $f:\mathbb{C}\to\mathbb{C}$ is entire, 
	then the series $\displaystyle \sum_{i=0}^\infty a_i \alpha^i$ converges for every element $\alpha \in \mathcal{A}$. 
	Indeed, the root test shows absolute convergence, and completeness of $\mathcal{A}$ ensures that the limit lies in $\mathcal{A}$.  
	
	With this in hand, we arrive at the following conclusion: if $\mathcal{A}$ is a unital associative Banach algebra over $\mathbb{C}$ 
	and the evaluation map $\operatorname{eva}_{f,\mathcal{A}}$ associated with an entire function $f$ is injective, 
	then necessarily $f$ must be affine linear, i.e., of the form $f(z) = az + b$ with $a \neq 0$.  
	
	In parallel with the complex field \( \mathbb{C} \), attention has also been given to the study of entire functions over \( \mathbb{C}_p \), the field of \( p \)-adic complex numbers. Both \( \mathbb{C} \) and \( \mathbb{C}_p \) are complete, algebraically closed fields of characteristic zero, each endowed with a nontrivial absolute value which is archimedean in the case of \( \mathbb{C} \), and nonarchimedean in the case of \( \mathbb{C}_p \). According to the classical result of Ostrowski \cite[Theorem 5.5]{Bo_Gri_07}, up to isomorphism preserving absolute values, \( \mathbb{R} \) and \( \mathbb{C} \) are the only complete fields with respect to an archimedean absolute value. Consequently, \( \mathbb{C} \) is, up to isomorphism, the only complete algebraically closed field with a nontrivial archimedean absolute value. In contrast, the landscape in the nonarchimedean case is significantly richer and more intricate; for a detailed discussion, see \cite[Section 9]{Pa_Ba_18}. Hence, we anticipate that this work may also be of interest to researchers in nonarchimedean analysis, as we believe that our results concerning entire functions continue to hold in the nonarchimedean setting.

	\section*{Acknowledgments}
	
	This research is funded by Vietnam National Foundation for Science and Technology Development (NAFOSTED) under grant number IZVSZ2\underline{~}229554.

	\section*{Declarations}
	
	Our statements here are the following:
	
	\begin{itemize}
		\item {\bf Ethical Declarations and Approval:} The authors have no any competing interest to declare that are relevant to the content of this article.
		\item {\bf Competing Interests:} The authors declare no any conflict of interest.
		\item  {\bf Authors' Contributions:} All three listed authors worked and contributed to the paper equally. The final editing was done by the corresponding author Frank Kutzschebauch and was approved by all of the present authors.
		\item {\bf Availability of Data and Materials:} Data sharing not applicable to this article as no data-sets or any other materials were generated or analyzed during the current study.
	\end{itemize}

	

	\bibliographystyle{amsplain}

\begin{thebibliography}{100}
		
		
		
		\bibitem{Pa_Ax_68}	J. Ax, The elementary theory of finite fields, \textit{Ann. of Math. (2)}   \textbf{88} (1968), 239–271.
		
		\bibitem{Pa_Av_13} A. Avilés, P. Koszmider, A Banach space in which every injective operator is surjective, \textit{Bull. Lond. Math. Soc.} \textbf{45} (5) (2013), 1065-1074. 
		
		
		
		
		\bibitem{Pa_Ba_18} A. Barría Comicheo and K. Shamseddine, \textit{Summary on nonarchimedean valued fields}, Advances in ultrametric analysis, Contemp. Math., vol. \textbf{704}, American Mathematical Society, Providence, RI, 2018, pp. 1–36. 
		
		
		\bibitem{Pa_BP_60}
		C.\ Bessaga,  A.\ Pe\l czy\' nski, 
		Banach spaces non-isomorphic to their Cartesian squares, \textit{Bull. Acad. Polon. Sci. S\' er. Sci. Math. Astr. Phys.} \textbf{8} (1960), 77-80.
		
		\bibitem{Pa_BuReSa_19} C. Buşe, D. O'Regan,  O. Saierli, A surjection problem leading to the Ax-Grothendieck theorem, \textit{Appl. Anal.} \textbf{98} (2019), no. 4, 843–850.
		
		\bibitem{Pa_BuReSa_21} C. Buşe, D. O'Regan,  O. Saierli, A surjectivity problem for 3 by 3 matrices, \textit{Oper. Matrices} \textbf{13} (2019), no. 1, 111–119; corrigendum \textit{Oper. Matrices} \textbf{15} (2021), no. 4, 1559–1561.
		
		
		
		
		
		
		
		\bibitem{Bo_Es_00} A. van den Essen, \textit{Polynomial Automorphisms and the Jacobian Conjecture}, in: Progress in Mathematics, vol. 190, Birkhäuser Verlag, Basel, 2000. 
		
		
		\bibitem{Bo_FrSi_08} A. Frommer, V. Simoncini, \textit{Matrix Functions. Model Order Reduction: Theory, Research Aspects and Applications.} The European Consortium for Math in Industry, Vol. 13, Springer, Berlin, 2008.
		
		
		\bibitem{Bo_Gri_07} P. A. Grillet, \textit{Abstract algebra}, 2nd ed., Graduate Texts in Mathematics, vol. \textbf{242}, Springer, New York, 2007.
		
		\bibitem{Pa_Grothendieck_66} A. Grothendieck, Éléments de géométrie algébrique. IV. Étude locale des schémas et des morphismes de schémas. III, \textit{Inst. Hautes Études Sci. Publ. Math.} \textbf{28} (1966), 255 pp.
		
		
		
		
		
		
		
		\bibitem{Bo_Li_97} R. Lidl, H. Niederreiter, \textit{Finite Fields}, Encyclopedia of Mathematics and Its Applications, vol. 20, Cambridge University Press, Cambridge, 1997.
		
		
		
		
		
		
		
		
		\bibitem{Bo_Re_91} R. Remmert, \textit{Theory of complex functions}, Graduate Texts in Mathematics, vol. 122, Springer-Verlag, New York, 1991. Translated from the second German edition by Robert B. Burckel.
		
		
		
		
		
		
		
		
		
	\end{thebibliography}

\end{document}